\documentclass[11pt]{amsart}
\input{cf.sty}
\usepackage{amssymb,caption}
\usepackage[margin=1.2in]{geometry}
   
\newcommand{\supp}[1]{\text{ supp }(#1)}

\begin{document}

\title{An $L^{2}-$stability estimate for periodic nonuniform sampling in higher dimensions}
\author{Christina Frederick}\email{}
\date{\today}
\address{}

\begin{abstract} We consider sampling strategies for a class of multivariate bandlimited functions $f$ that have a spectrum consisting of disjoint frequency bands. Taking advantage of the special spectral structure, we provide formulas relating $f$ to the samples $f(y), y\in X$, where $X$ is a periodic nonuniform sampling set. In this case, we show that the reconstruction can be viewed as an iterative process involving certain Vandermonde matrices, resulting in a link between the invertibility of these matrices to the existence of certain sampling sets that guarantee a unique recovery. Furthermore, estimates of inverse Vandermonde matrices are used to provide explicit $L^{2}$-stability estimates for the reconstruction of this class of functions. 

\end{abstract}

%
%


\maketitle

\section{Introduction}

Almost all scientific applications that involve processing large datasets or images rely heavily on the acquisition and storage of measurements of a complex physical process or system. Efficiency is gained when the samples contain just enough information to uniquely recover the signal, but no more. As advancements in technology demand more sophisticated methods to handle data, it becomes even more important to understand the fundamental structure of sampling processes. 

Mathematically, the problem can be formulated as the reconstruction of a function $f$, in a reliable manner, from knowledge of a collection of its function values, called samples, evaluated at points in a discrete sampling set $X=\{x_{j}\}$. Shannon's sampling theorem states that bandlimited functions $f$ are uniquely determined by its samples taken on the uniform sampling set $X=\{j \Delta\mid j\in \ZZ\}$, where the spacing $\Delta$ must be chosen to resolve the highest frequencies in the spectrum \cite{Shannon}. This theory also provides explicit reconstruction formulas that relate the samples $f(y), y\in X$ to the desired function $f$. 

Sampling of bandlimited functions using arbitrary sampling sets is less understood, especially in higher dimensions, and there is an enormous volume of literature on nonuniform sampling theory and its various generalizations and extensions \cite{Behmard2002,Benedetto2001, Jerri1977, Marvasti2001, Papoulis1977,Strohmer2006}.  It is important to also assess the sensitivity of the reconstruction with respect to to small changes in the sampling set. A set $X\subset \RR$  is a stable set of sampling for $\mathcal{B}(\Omega)$ if for some constant $C>0$,
\[\|f\|^{2}_{L_{2}(\RR)}\leq C \sum_{y\in X}|f(y)|^{2},\quad \text{for all } f\in \mathcal{B}(\Omega).\]
Explicit estimates of the stability constant for arbitrary sampling sets $X\subset \RR^{d}$ are given in \cite{Grochenig2001}, however the proof relies on an upper bound on the maximum distance between a sampling point $x_{j}$ and its nearest neighbor $x_{j'}\neq x_{j}$. This upper bound depends on the dimension, and the given estimates deteriorate in higher dimensions. According to \cite{Beurling1967, Landau}, the necessary density conditions on the sampling set are independent of dimension. Density conditions and explicit stability bounds are given in \cite{Adcock2016} for bunched sampling of multivariate functions and their derivatives.

In this paper, we determine the stability constants explicitly by examining properties of matrices that are produced from periodic sampling formulas and applying estimates of inverse Vandermonde matrices. Here, additional conditions are imposed on the spectral structure of multivariate bandlimited functions. These functions are well-studied in the literature on periodic sampling, however stability estimates are rarely given explicity. In \S\ref{sec:bps}, we state the main result and give lemmas needed for the main reconstruction result and estimates of inverse Vandermonde matrices. In \S \ref{sec:iterative-recon}, an iterative reconstruction algorithm is given. The proof of the main stability result is given in \S\ref{sec:stability-proof}.

\subsection*{Notation} In this paper, $L^{2}(\RR^{d})$ is the usual Hilbert space of measurable, square-integrable functions, (i.e. $\|f\|_{L^{2}(\RR^{d})}=(\int_{\RR^{d}} |f(x)|^{2}dx )^{1/2}<\infty$) and we denote the Fourier transform of $f\in L^{2}(\RR^{d})$  by $\hat{f}(\xi)=\int_{\RR^d} f(x)e^{-2\pi i \ip{x}{\xi}}dx$, where $\ip{x}{\xi} = \sum_{i=1}^{d}x_{i}\xi_{i}$ is the Euclidean inner product on $\RR^{d}$.   The maximum norm is $\|x\|_{\infty} = \max_{i=1, \hdots d} |x_{i}|$. For a bounded, measurable set $\Omega\subset\RR^{d}$, the space of multivariate bandlimited functions $\mathcal{B}(\Omega)$ is defined as $\mathcal{B}(\Omega)=\{f\in L^{2}(\RR^d) \mid \supp{\hat{f}}\subset \Omega\}.$

\section{Main result}\label{sec:bps}
Motivated by the one-dimensional case presented in \cite{Engquist2014}, the setting here involves multivariate functions that are bandlimited to a set composed of a finite number of nonoverlapping translations of $\Omega= (-1/2, 1/2)^{d}$. For a nonnegative integer $M$, the set $\mc{M}^{d}=\ZZ^{d}\cap[-M, M]^{d}$ is taken to be a bounded, discrete set corresponding to the disjoint frequency bands in the spectrum $\Omega_{M,  N}\subset \RR^{d}$, given by
\begin{align}
\Omega_{M,  N}=\{\xi+mN \mid \xi\in\Omega, m\in\mc{M}^{d}  \}, \qquad N\geq 1\labeleq{OmegaMN}.
\end{align}
The left plot in Figure \ref{fig:mbspectra} illustrates $\Omega_{M,N}$ in two dimensions.

The main theorem describes the reconstruction of functions $f\in \mc{B}(\Omega_{M,N})$ from samples on the periodic nonuniform sampling set
\begin{align}
X&= \{j\Delta+ k\delta\Delta\mid j \in \ZZ^{d}, k\in \mathcal{K}^{d}\}, 
\labeleq{X}
\end{align}
where $\mathcal{K}^{d}=\{ k\in \ZZ^{d}\mid 0\leq |k|\leq 2M\}$, and sufficient conditions on $\delta$ and $\Delta$ that guarantee a unique and stable reconstruction are given in the main result.

\begin{theorem}\label{thm:main-samp-nd}
A function $f\in\mc{B}(\Omega_{M,N})$ can be uniquely reconstructed from the periodic nonuniform samples 
$ f(y), y\in X$, where $X$ is given by \refeq{X} and $\frac{1}{N}\leq \Delta\leq1$ and $0<\delta\leq\frac{1}{(2M+1)N}$. In addition, the following stability estimate holds:
\begin{align}
A \|f\|_{L^{2}(\RR^{d})}^{2} \leq \Delta^{d}\sum_{y\in X}|f(y)|^{2}\leq B \|f\|_{L^{2}(\RR^{d})}^{2},\labeleq{thmstability}
\end{align}
where the constants $A$ and $B$ satisfy
 \begin{align}
{(2M+1)^{-d}} \(\prod\limits_{m=1}^{2M}{\sin\(m\pi \delta  \)}\)^{2d}\leq A\leq B,\qquad (2M+1)^{d}\leq B   \leq (2M+1)^{2d}.\labeleq{thmab}
\end{align}
When $N=1$ and $\delta=1/(2M+1)$, the uniform sampling set $X$ allows for reconstruction with stability constants
\begin{align}
A = B=(2M+1)^{d}.\labeleq{tightframe}
\end{align}

\end{theorem}

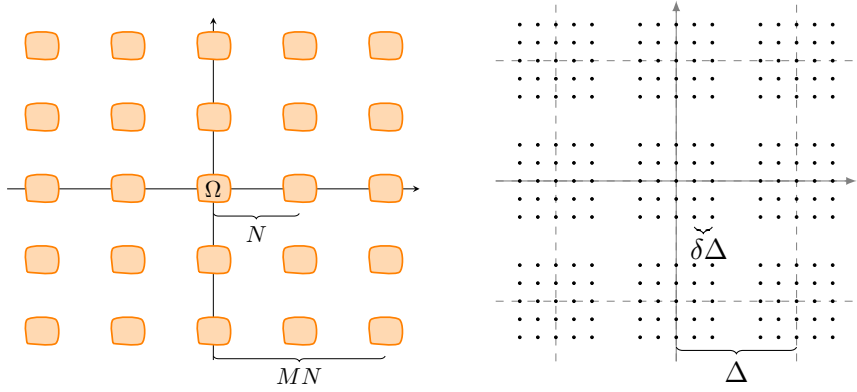
\begin{figure}
\captionsetup{width=\textwidth}
 \caption{On the left is the multiband spectrum given by \refeq{OmegaMN}, and on the right is a  periodic nonuniform sampling set \refeq{X}. } \label{fig:mbspectra}
\begin{center}
\begin{tikzpicture}[>=stealth,scale=.8]
    \begin{axis}[
        axis x line=middle,
        axis y line=middle,
        xmin=-6,     
        xmax= 6,    
        ymin=-6,     
        ymax= 6,   
        ticks=none,
        enlargelimits=true,
        after end axis/.code={
            \draw [decorate,decoration={brace,mirror,raise=6pt}] (axis cs:0,-.5) -- (axis cs:3,-.5) node [midway,below=8pt, black] {$N$};
                        \draw [decorate,decoration={brace,mirror,raise=6pt}] (axis cs:0,-6.5) -- (axis cs:6,-6.5) node [midway,below=8pt, black] {$MN$};
      \node[black] at (axis cs:0,0) {$\Omega$};
        }]

\foreach \x in {-6, -3,0,3,6}{
\foreach \y in {-6, -3,0,3,6}{
        \addplot[mark=none, orange, smooth, thick, fill=orange!30] coordinates {(-.5+\x,-.5+\y) (-.5+\x, .5+\y) (.5+\x,0.5+\y)  (.5+\x, -.5+\y) (-.5+\x,-.5+\y)};
        }
         }

    \end{axis} 
\end{tikzpicture}\hspace{1cm}\begin{tikzpicture}[scale=.8]
    \coordinate (Origin)   at (0,0);
    \coordinate (XAxisMin) at (-3,0);
    \coordinate (XAxisMax) at (3,0);
    \coordinate (YAxisMin) at (0,-3);
    \coordinate (YAxisMax) at (0,3);
    \draw [thin, gray,-latex] (XAxisMin) -- (XAxisMax);
    \draw [thin, gray,-latex] (YAxisMin) -- (YAxisMax);

    \draw[style=help lines,dashed] (-3,-3) grid[step=2cm] (3,3);

    \foreach \xj in { -1, 0, 1} {
     \foreach \yj in {-1, 0, 1}{
     \foreach \x in {-2,...,2}{
      \foreach \y in {-2,...,2}{
        \node[draw,circle,inner sep=.1pt,fill] at (.3*\x+2*\xj,.3*\y+2*\yj) {};
            
      }
    }
   }
    }
 \draw [decorate,decoration={brace,mirror,raise=6pt}] (.3,-.5) -- (.6,-.5) node [anchor=north, black] at (.55, -.8) {$\delta\Delta$};
                        \draw [decorate,decoration={brace,mirror,raise=6pt}] (0,-2.5) -- (2,-2.5) node [midway,below=8pt, black] {$\Delta$};

  \end{tikzpicture}

\end{center}
\end{figure}

As a remark, pointwise evaluation is well-defined and  $\refeq{thmstability}$
is also a statement about the boundedness of the sampling operator and its invertibility. Figure \ref{fig:mbspectra} provides an illustration of a periodic nonuniform sampling set of the form \refeq{X} in two dimensions. Landau proved a lower bound on the sampling rate needed for stable sampling \cite{Landau1967}. Applied to periodic nonuniform sampling, this lower bound is $|\Omega_{MN}| = (2M+1)^{d}$ samples in each interval $z+[0,1)^{d}, z\in \ZZ^{d}$. The sampling sets in Theorem \ref{thm:main-samp-nd} achieve this rate. It should be noted that the lower bound on $\Delta$ is not necessary for sampling and reconstruction, however the added restriction removes the possibility of oversampling.

\subsection{Subsampling formula} \label{sec:mb-samp}

We will build a sampling operator as in \cite{Behmard2002, Behmard2006} using the continuous function $\varphi_{\Delta}$ defined by
\begin{align}
\varphi_{\Delta}(z) &= \Delta^{d}\int_{\Omega_{\Delta}} e^{2\pi i \ip{z}{\xi} }d\xi, \quad \Omega_{\Delta}=\left(-\frac{1}{2\Delta}, \frac{1}{2\Delta}\right)^{d}.\end{align}
Since $\widehat{\varphi_{\Delta}}(\xi) = \Delta^{d}$ if $\xi\in\Omega_{\Delta}$ and $\widehat{\varphi_{\Delta}}=0$ otherwise, it follows that $\varphi_{\Delta}\in\mc{B}(\Omega_{\Delta})$ and $\|{\varphi_{\Delta}}\|_{L^{2}(\RR^{d})} = \sqrt{\Delta^{d}}$. 

Let $X_{0} =\Delta\ZZ^{d}$ and $X_{k} = X_{0}+k\delta\Delta$ for $k\in \mathcal{K}^{d}$. The sampling operator $S_{X_{k}}$ is formally defined for functions $g\in L^{2}(\RR^{d})$ by the formula
\begin{align}
S_{X_{k}}g(x)&= \sum_{y\in X_{k}} g(y) \varphi_{\Delta}(x-y)\labeleq{SX}.\end{align}

A higher dimensional extension of the classical sampling theorem \cite{Behmard2006, Behmard2002} states that if $\Omega\subseteq\Omega_{\Delta}$, then $S_{X_{k}}g$ is well-defined for $g$ in $\mc{B}(\Omega)$ and $S_{X_{k}}g = g$. Furthermore, the following stability estimate holds,
\begin{align}
\|S_{X_{k}}g\|_{L^{2}(\RR^{d})}^{2}&=\Delta^{d} \sum\limits_{y\in X_{k} }|g(y) |^{2}.\labeleq{Shannonstability}
\end{align}

Our main result involves subsampling by applying the sampling operator to functions $f\in\mathcal{B}(\Omega_{M,N})$. To do this, we first represent $f$ in terms of $\Omega-$bandlimited functions.

\begin{lemma}\label{lemma:fl2} For each $f\in\bl{\Omega_{M, N}}$ there exist functions $c_{m} \in \bl{\Omega}$ so that  \begin{align}
f(x) = \sum_{m\in\mc{M}^{d}}c_m\(x\)e^{2\pi i \ip{mN}{x}},\labeleq{twoscale}\end{align} and 
\begin{align*}
\|f\|_{L^{2}({\RR}^{d})}^{2} = \sum_{m\in \mc{M}^{d}}\|c_{m}\|_{L^{2}({\RR}^{d})}^{2}.
\end{align*}
\end{lemma}

\begin{proof} For each $m\in \mathcal{M}^{d}$, define $c_{m}$ by 
\begin{align*}
c_{m}(x)=\int_{\RR^{d}}\left[\hat{f}(\xi)\mathbbm{1}_{\Omega}(\xi-mN)  \right]e^{2\pi i \ip{\xi}{x}}d\xi
\end{align*}
where $\mathbbm{1}_{\Omega}(\xi)=1$ for  $\xi\in {\Omega}$ and $0$ otherwise. Then, $\widehat{c_{m}}(\xi) = \hat{f}(\xi+mN)\mathbbm{1}_{\Omega}(\xi)$  and $f$ has the representation \refeq{twoscale}. Taking the Fourier transform,
\begin{align*}
\|f\|_{L^{2}(\RR^{d})}^{2}=\int_{\RR^{d}}\left\vert\hat{f}(\xi)\right\vert^{2} d \xi &=\sum_{m\in \mc{M}^{d}} \int_{\Omega} \left\vert\hat{f}(\xi + mN)\right\vert^{2} d \xi
=\sum_{m\in \mc{M}^{d}} \int_{\RR^{d}} \left\vert\widehat{c_{m}}(\xi )\right\vert^{2} d \xi.
\end{align*}
\end{proof}

To account for the case where integer translates of $\frac{1}{\Delta}$ are not necessarily integer multiples of $N$, for each $\xi\in \Omega_{\Delta}$, we let $z_{m}(\xi)$ be the unique vector in $\Omega_{\Delta}^{-1} = \frac{1}{\Delta}\ZZ^{d}$ for which 
\begin{align}
\xi + z_{m}(\xi) \in \(\Omega+mN\)\labeleq{mN}.
\end{align}
Set $L_{m}(\xi)\in\ZZ^{d} = z_{m}(\xi)\Delta$ and $\alpha_{m}(\xi) = mN-z_{m}(\xi)\in\Omega_{\Delta}$.
Since $|\Omega|<|\Omega_{\Delta}|$, it follows that the set for each $m\in\mc{M}^{d}$, there are at most $2^{d}$ vectors in the set $\mathcal{Z}_{m}\subset \Omega_{\Delta}^{-1}$ so that $\(\Omega+mN\)\subseteq \cup_{z\in \mathcal{Z}_{m}}\(\Omega_{\Delta}+z\)$.
\begin{lemma}\label{lem:SXCmk} Let $c_{m}\in\mc{B}(\Omega)$ and define $\tilde{c}_{m}(x):={c}_{m}e^{2\pi i \ip{mN}{x}}$. Then, 
\begin{align}
\widehat{SX_{k}\tilde{c}_{m}}(\xi)=
{\hat{c}_{m}}(\xi - mN+z_{m}(\xi)) e^{ 2\pi i \ip{z_{m}(\xi)\Delta}{k\delta}} \mathbbm{1}_{\Omega_{\Delta}}(\xi)\labeleq{cmk}.
\end{align}
Furthermore, $SX_{k}\tilde{c}_{m}(y') = \tilde{c}_{m}(y')$ for all $y'\in X_{k}$
\end{lemma} 

\begin{proof} Applying the formula \refeq{Shannonstability} to $c_{m}\in \mathcal{B}(\Omega)$, it follows that $\sum_{y\in X_{k}}|c_{m}(y)e^{2\pi i \ip{mN}{y}}|^{2}=\sum_{y\in X_{k}}|c_{m}(y)|^{2}<\infty$. By Cauchy-Schwarz, the function defined by $\refeq{cmk}$ is well-defined and square-integrable. 
Then,
\begin{align}
\tilde{c}(x)&= \sum_{y\in X_{k}} c_{m}(y)e^{2\pi i \ip{mN}{y}} \varphi_{\Delta}(x-y) \nonumber\\
&= \sum_{n\in \ZZ^{d}} c_{m}((n+k\delta)\Delta)e^{2\pi i \ip{mN}{(n+k\delta)\Delta}} \varphi_{\Delta}(x-(n+k\delta)\Delta) \nonumber\\
&= \Delta^{d}\int_{\Omega_{\Delta}} \sum_{n\in \ZZ^{d}} c_{m}((n+k\delta)\Delta)e^{2\pi i \ip{mN}{(n+k\delta)\Delta}}e^{2\pi i\ip{x-(n+k\delta)\Delta}{\xi}}d\xi \nonumber\\
&= \int_{\Omega_{\Delta}} \sum_{l\in \ZZ^{d}} e^{2\pi i \ip{ k\delta}{l}}\( \int_{\RR^{d}}  c_{m}(y)e^{-2\pi i\ip{y}{\xi-mN+l/\Delta}}dy \)e^{2\pi i\ip{x}{\xi}} d\xi  \labeleq{Poisson}\\
&= \int_{\Omega_{\Delta}} \sum_{z_{m}\in\Omega_{\Delta}^{-1}} e^{2\pi i \ip{ k\delta}{z_{m}(\xi)\Delta}}\( \int_{\RR^{d}}  c_{m}(y)e^{-2\pi i\ip{y}{\xi-mN+z_{m}(\xi)}}dy \)e^{2\pi i\ip{x}{\xi}} d\xi 
\end{align}
The higher dimensional Poisson summation formula with results in \refeq{Poisson}. Then, taking the Fourier Transform,
\begin{align}
\widehat{SX_{k}\tilde{c}_{m}}(\xi)
&=  \widehat{c_{m}}\(\xi-mN+z_{m}(\xi)\) e^{2\pi i \ip{z_{m}(\xi)\Delta}{ k\delta}} \mathbbm{1}_{\Omega_{\Delta}}(\xi). \nonumber
\end{align}

Then, for $y'\in X_{k}$
\[SX_{k}\tilde{c}_{m}(y')=\sum_{y\in X_{k}}\tilde{c}(y) \varphi_{\Delta}(y-y') = \tilde{c}_{m}(y').\]
\end{proof}

The next lemma describes the reconstruction of functions $f\in\mc{B}(\Omega_{M,  N})$ and is a generalization of the one-dimensional result in \cite{Engquist2014}. 
\begin{lemma} Let $f$ be a function in the space $\mc{B}(\Omega_{M,  N})$ and let $X_{k}$ be a sampling set of the form \refeq{X}. The function $S_{X_{k}}f(x)= \sum_{y\in X_{k}} f(y) \varphi_{\Delta}(x-y)$, is in $\bl{\Omega_{\Delta}}$, and,
\begin{align}
\|SX_{k}f\|_{L^{2}(\RR^{d})}^{2}= \Delta^{d} \sum\limits_{y\in X_{k}}|f(y)|^{2}.\labeleq{SXkfstability}
\end{align}

\label{lemma:SX}
\end{lemma} 

\begin{proof} 
	
Applying the linear sampling operator to the functions $\tilde{c}_{m}$ defined in Lemma \ref{lem:SXCmk}, it follows that
\begin{align}
\widehat{SX_{k}f}(\xi) = \sum_{m\in\mc{M}^{d}} \widehat{c_{m}}(\xi-mN+z_{m}(\xi)) e^{2\pi i \ip{z_{m}(\xi)\Delta}{ k\delta}} \mathbbm{1}_{\Omega_{\Delta}}(\xi)  \labeleq{fk}.
\end{align}
Therefore, $SX_{k}f\in \mathcal{B}(\Omega_{\Delta})$ and for $y'\in X_{k}$,  $SX_{k}f(y') = \sum_{m\in\mc{M}^{d}}SX_{k}\tilde{c}_{m}(y') = \sum_{m\in\mc{M}^{d}}\tilde{c}_{m}(y') = f(y')$. Finally, by \refeq{Shannonstability},

\begin{align}
\|SX_{k}f\|^{2}_{L^{2}(\RR^{d})} &=\Delta^{d}\sum_{y\in X_{k}}|S{X}_{k}f(y)|^{2}=\Delta^{d}\sum_{y\in X_{k}}|{f}(y)|^{2}.
\end{align}

\end{proof}

\subsection{Vandermonde matrix estimates for $d=1$}\label{sec:Vandermonde}
In \cite{Engquist2014}, the following result is given for the system of equations \refeq{fk} in one spatial dimension. We present the proof, reformulated in the present context.

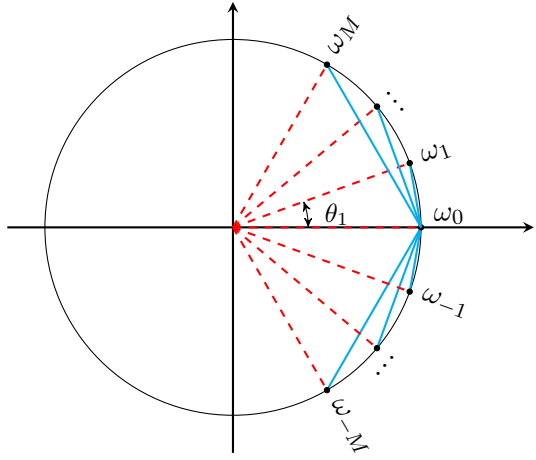
\begin{figure}
\captionsetup{width=\textwidth}
\caption{Distribution of the complex numbers $\omega_{-M}, \hdots \omega_{M}$ on the unit circle produced by $\omega_{m} = e^{i \theta_{m}}$. The blue chords represent the distances in the products in \refeq{prodwm}. }\label{fig:gautschi}
\begin{center}
        \begin{tikzpicture}[scale=1]
    \draw[thick,-stealth,black] (-3,0)--(4,0) coordinate (A) node[below] {}; 
    \draw[thick,-stealth,black] (0,-3)--(0,3) node[left] {}; 
    \draw[black,thin] (0,0) circle (2.5cm);

 \draw[rotate=60] (3,0) node[midway,right=1in] { $\omega_{M}$};
 \draw[rotate=40] (3,0) node[midway,right=1in] { $\vdots$};
 \draw[rotate=20] (3,0) node[midway,right=1in] { $\omega_{1}$};
 \draw[rotate=3] (3,0) node[midway,right=1in] { $\omega_{0}$};
 \draw[rotate=-20] (3,0) node[midway,right=1in] { $\omega_{-1}$};
 \draw[rotate=-40] (3,0) node[midway,right=1in] { $\vdots$};
 \draw[rotate=-60] (3,0) node[midway,right=1in] { $\omega_{-M}$};
  
 \draw [<->] ($(1,0)$) 
    arc (0:20:1) node[midway, right=.1cm] {\small $\theta_{1}$};
 
\foreach \theangle in {-60,-40, -20,0,20,40, 60}
{

     \draw[thick, cyan] (\theangle:2.5cm) -- (2.5,0) node[midway,right] {}; 
      \filldraw[black] (\theangle:2.5cm) circle(1pt);
    \draw[thick,dashed,red,-,rotate=\theangle] (0,0) -- (2.5,0) coordinate (B); 
   
}

    \end{tikzpicture}

\end{center}
\end{figure}
\begin{lemma}\label{lem:vinv}Let  $d=1$, and define $F(\xi) = (\widehat{SX_{0}f}(\xi), \hdots, \widehat{SX_{2M}f}(\xi))$, where $X_{k}=\{j\Delta x +k\delta\Delta \mid j\in \ZZ\}$, where $\frac{1}{N}\leq \Delta \leq 1$ and $0<\delta\leq1/{(2M+1)N}$. The system of equations given by \refeq{fk} has the form \[V\tilde{C}=F, \qquad \tilde{C}(\xi) = \({C_{-M}}(\xi),\hdots {C_{M}}(\xi)\)\]  where  $V$ is an invertible $(2M+1)\times (2M+1)$ Vandermonde matrix and ${C_{m}}(\xi)=\widehat{c_{m}}(\xi - mN+z_{m}(\xi))$. Then, \begin{align}
\frac{1}{2M+1}< \|V^{-1}\|_\infty\leq \prod\limits_{m=1}^{2M}\frac{1}{\sin\(m \pi \delta \)} . \labeleq{Vest}
\end{align}
\end{lemma}
\noindent Furthermore, if $N =1$ and $\delta=\frac{1}{(2M+1)}$, then $\|V_n^{-1}\|_\infty=1$.
\begin{proof}

The the system of equations \refeq{fk} in the one-dimensional case  is understood using properties of $n\times n$ Vandermonde matrices of the form
\begin{align*}
V_{n} =\begin{pmatrix}
1 &1 &\hdots& 1\\
\omega_1 &\omega_2 &\hdots& \omega_n\\
\omega_1^2 &\omega_2^2 &\hdots& \omega_n^2\\
\vdots &\vdots &\hdots& \vdots\\
\omega_1^{n-1} &\omega_2^{n-1} &\hdots& \omega_n^{n-1}\\
\end{pmatrix},
\end{align*} where $\omega_1, \hdots, \omega_n$ are distinct nonzero complex numbers and $n>1$. 
A result of Gautschi \cite{Gautschi1990} states that 
\begin{align}
\max_{m}\prod\limits_{m'\neq m}\frac{1}{|\omega_{m'}-\omega_{m}|}<\|V_n^{-1}\|_\infty\leq \max_{m} \prod\limits_{m'\neq m}\frac{2}{|\omega_{m'}-\omega_{m}|}. \labeleq{gautschi}
\end{align}
We will fix $\xi\in \Omega_{\Delta}$ and suppress it from the notation. For $m = -M, \hdots, M$, and define the unit complex numbers $\omega_m= e^{i\theta_{m}}$ corresponding to the angles $\theta_{m}= 2\pi L_{m} \delta$.  Since $0=\theta_{0}<\theta_{1}< \hdots< \theta_{M}\leq 2\pi {M}/{(2M+1)}<\pi$ are distinct, the $(2M+1)\times (2M+1)$ Vandermonde matrix $\{(V)_{jm} =\omega_{m}^{j}\}_{0\leq j\leq 2M, -M\leq m\leq M}$ is invertible. 

For $N=1$ and $\delta=\frac{1}{2M+1}$,  $\theta_{m}=\frac{m}{2M+1}$ and $\omega_{m}=e^{ i \theta_{m}}$ are roots of unity and it is known that  $\|V^{-1}\|_\infty=1$. Otherwise, for $-M\leq m\leq M-1$, it follows that $N=\frac{L_{m+1}-L_{m}}{\Delta} +(\alpha_{m+1}-\alpha_{m})= \frac{L_{1}}{\Delta}+\alpha_{1}$. Therefore $L_{m+1}-L_{m}\geq L_{1}$  and $|\omega_{m+1}-\omega_{m}|\geq |\omega_{1}-\omega_{0}|$. 
 Furthermore, $\theta_{M}+\theta_{1}/2=\pi (2M+1)N\delta \leq\pi$ and so $|\omega_{M} - \omega_{-M}|=|1 - e^{ 2 i \theta_{M}}|  \geq |\omega_{1}-\omega_{0}|$. 
The product in \refeq{gautschi} is maximized when $m=0$, (see Figure \ref{fig:gautschi} for an illustration) and
\begin{align}
| \omega_{m}-\omega_{0}| = 2 |\sin\(\theta_{m}/{2}\)|= 2 |\sin\({\pi L_{m}\delta}\)|.\labeleq{prodwm}
\end{align}

The result \refeq{Vest} follows from the estimate (note $m\leq \lfloor{m}N \Delta\rfloor\leq mN$),
\begin{align}
\prod\limits_{m=1}^{2M}\frac{2}{|\omega_{m}-\omega_{0}|}& =\prod\limits_{m=1}^{2M}\frac{1}{\sin\(\lfloor{m}N \Delta\rfloor \pi  \delta \)} \leq \prod\limits_{m=1}^{2M}\frac{1}{\sin\({m} \pi  \delta \)}  \labeleq{maxest1}.
 \end{align}

 The product of the terms $|\omega_{m}-\omega_{0}|$ is minimized if $\delta=1/{(2M+1)N}$ and the lower bound is a well-known result for the product of the lengths of the sides of a regular polygon inscribed in the unit circle,
\begin{align}
\prod\limits_{m=1}^{2M}\frac{1}{{|\omega_{m}-\omega_{0}|}}\geq \prod\limits_{m=1}^{2M}\frac{1}{{2|\sin( mN \pi \delta)|}}\geq \prod\limits_{m=1}^{2M}\frac{1}{2\sin\(2\pi {\frac{m }{2M+1}}  \)}  \labeleq{maxest2}
=\frac{1}{2M+1}.
\end{align}

The upper bound in \refeq{gautschi} is attained if the complex numbers $\omega_{m}=|\omega_{m}|e^{ i\theta}$ are points on the same ray through the origin. In the present setting, this happens when $M=0$. 

\end{proof}

In higher dimensions, the invertibility of Vandermonde matrices is not guaranteed. The problem of efficiently characterizing the invertibility of these matrices is a challenging open problem. In the next section we reduce the full $d-$dimensional linear system to a sequence of invertible Vandermonde systems.

\section{Reconstruction algorithm}\label{sec:iterative-recon}

There are other general results for reconstructing bandlimited functions from sampling sets \cite{Faridani1994}. The purpose of the algorithm presented here is to assist with proving the stability estimates in Theorem \ref{thm:main-samp-nd}. For $d>1$, an iterative process can be used to determine unknown coefficients in \refeq{twoscale}. 

\begin{theorem}
The function $f\in \mc{B}(\Omega_{M,  N})$ is uniquely reconstructed from samples $f(y), y\in X$, where $X$ satisfies the assumption of Theorem \ref{thm:main-samp-nd}, using the following procedure for each $\xi\in \Omega_{\Delta}$,
\begin{description}
\item[\bf Step 1] For each $j \in \mc{K}^{d}$, define $F_{j}(\xi) = \widehat{S_{X_{j}}f}(\xi)$. 
\item[\bf Step 2] If $d=1$, set $\tilde{C}(\xi):=\begin{pmatrix}  F_{0}(\xi) , \hdots, F_{2M}(\xi) \end{pmatrix}$, $F^{m_{d}}(\xi):=(V^{-1}\tilde{C}(\xi))_{m_{d}}$,  $m_{d}\in \mc{M}^{1}$. Then, skip to Step 5.

If $d\geq 2$, for each $j' \in \mc{K}^{d-1}$,  \begin{align*}\tilde{C}_{j'}(\xi) &:=\(  F_{(j', 0)}(\xi) , \hdots, F_{(j', 2M)}(\xi) \)\\
F_{j'}^{m_{d}}(\xi)&:=(V^{-1}\tilde{C}_{j'}(\xi))_{m_{d}}, \qquad m_{d}\in \mc{M}^{1}.
\end{align*}

\item[\bf Step 3] For $l=d-1, \hdots, 2$, and for all  $j'\in \mc{K}^{l-1}, m'\in \mc{M}^{d-l}$, 
\begin{align*}\tilde{C}_{j'}^{m'}(\xi) &:=\( F^{m'}_{(j',0)}(\xi) , \hdots ,  F^{m'}_{(j',2M)}(\xi) \)\\
F_{j'}^{(m_{l}, m')}(\xi) &:=(V^{-1}\tilde{C}_{j'}^{m'}(\xi))_{m_{l}},  \qquad m_{l}\in \mc{M}^{1}.
\end{align*}

\item[\bf Step 4] For each $ m'\in \mc{M}^{d-1}$, define 
\begin{align*}\tilde{C}^{m'}(\xi) &:=\(  F^{m'}_{0}(\xi), \hdots , F^{m'}_{2M}(\xi) \)\\
 F^{(m_{1}, m')}(\xi)&:=(V^{-1}\tilde{C}^{m'}(\xi))_{m_{1}},  \qquad m_{1}\in \mc{M}^{1}.
\end{align*}

\item[\bf Step 5] The reconstruction of $\hat{f}$ is given by the formula
\begin{align*}
 \hat{f}(\xi+z_{m}(\xi))=\sum\limits_{m\in \mc{M}^{d}}F^{m}(\xi).
\end{align*}

\end{description}
\label{thm:reconstruction}
\end{theorem}
\begin{proof}
If $d=1$, then Lemma \ref{lem:vinv} guarantees the unique recovery of the coefficient functions $
F^{m}(\xi):=\widehat{c_{m}}(\xi - mN+z_{m}(\xi)),  m\in \mc{M}^{1}$. Then, as in \cite{Engquist2014}, $\hat{f}(\xi+z_{m}(\xi))$ can be uniquely reconstructed using the formula \begin{align*}
F^{m}(\xi) =\widehat{c_m}\(\xi -mN+ z_{m}(\xi)\) = \hat{f}(\xi+z_{m}(\xi)).
\end{align*}
Let $L_{m} = L_{m}(\xi)\in \ZZ^{d}$.

Suppose $d>1$. Then, for each fixed $j' \in \mc{K}^{d-1}$ and $k_{d}\in \mc{K}^{1}$, the system of equations  \refeq{fk} corresponding to the sampling sets $\Lambda_{(j',k_{d})}$ and coefficient indices $m=(n', m_{d})\in \mc{M}^{d}$ can be expressed
\begin{align}
F_{(j',k_d)}(\xi)=\sum\limits_{m_{d}\in \mc{M}^{1}}\(\sum\limits_{n'\in \mc{M}^{d-1}}\hat{c}_{(n',m_{d})}\(\xi -mN+ z_{m}(\xi)\)e^{2\pi i \ip{L_{n'}}{j'\delta}  }\)e^{2\pi i {L_{m_{d}}}{k_{d}}\delta  }.\labeleq{step2eq}
\end{align}
The system \refeq{step2eq} is an invertible Vandermonde system with the unknowns \begin{align*}
F_{j'}^{m_{d}}(\xi):=\sum\limits_{n'\in \mc{M}^{d-1}}\hat{c}_{(n',m_{d})}\(\xi -mN+ z_{m}(\xi)\)e^{2\pi i \ip{L_{n'}}{j'\delta}  }, \qquad {m_{d}}\in \mc{M}^{1}.
\end{align*}
This process is then repeated for $l=d-1, \hdots, 2$. For each, $j'\in \mc{K}^{l-1}, k_{l}\in \mc{K}^{1}, m'\in \mc{M}^{d-l}$, the system of equations corresponding to $F_{(j',k_{l})}^{m'}$and coefficient indices $m=(n',m_{l},m')\in \mc{M}^{d}$ is 
\begin{align*}
F_{(j',k_{l})}^{m'}(\xi)=\sum\limits_{m_{l}\in \mc{M}^{1}}\(\sum\limits_{n'\in \mc{M}^{l-1}}\hat{c}_{(n',m_{l},m')}\(\xi -mN+ z_{m}(\xi)\)e^{2\pi i \ip{L_{n'}}{j'\delta}  }\)e^{2\pi i {L_{m_{l}}}{k_{l}\delta} },
\end{align*}
which is an invertible Vandermonde system with the unknowns \begin{align*}
F_{j'}^{(m_{l},m')}(\xi):=\sum\limits_{n'\in \mc{M}^{l-1}}\hat{c}_{(n',m_{l},m')}\(\xi -mN+ z_{m}(\xi)\)e^{2\pi i \ip{L_{  n'}}{j'\delta}  }, \qquad m_{l}\in \mc{M}^{1}.
\end{align*}
In the final step, for each $m'\in \mc{M}^{d-1}$ and $k_{1}\in \mc{K}^{1}$, the system of equations corresponding to $F_{k_{1}}^{m'}$ and coefficient indices $(m_{1},m')\in \mc{M}^{d}$ is 
\begin{align*}
F_{k_{1}}^{m'}(\xi)=\sum\limits_{m_{1}\in \mc{M}^{1}}\hat{c}_{(m_{1},m')}\(\xi -mN+ z_{m}(\xi)\)e^{2\pi i {L_{ m_{1}}}{k_{1}\delta} },
\end{align*}
which is an invertible Vandermonde system with the unknowns \begin{align*}
F^{(m_{1},m')}(\xi):=\hat{c}_{(m_{1},m')}\(\xi -mN+ z_{m}(\xi)\), \qquad m_{1}\in \mc{M}^{1} = \hat{f}(\xi+z_{m}(\xi)).
\end{align*}

Then, for all $m\in\mc{M}^{d}$, $\hat{f}({\xi}+z_{m}(\xi))$ can be reconstructed using the formula \begin{align*}
F^{m}(\xi) =\hat{c}_m\(\xi-{mN}+z_{m}(\xi)\) = \hat{f}(\xi+z_{m}(\xi)). \end{align*}
The reconstruction formula is unique due to the invertibility of the Vandermonde system in each step. Since, $\Omega_{M,N}\subseteq \cup_{m\in\mc{M}^{d}}\{\xi+z_{m}(\xi)\mid \xi\in\Omega_{\Delta}\}$, $\hat{f}$ is completely determined on $\Omega_{M,N}$. By taking the inverse Fourier Transform, this is equivalent to the unique reconstruction of $f$.

\end{proof}

\section{Proof of Theorem \ref{thm:main-samp-nd}}\label{sec:stability-proof}

The proof of the main result involves an analysis of each step in the algorithm given in Theorem \ref{thm:reconstruction}.
\begin{proof}
By Lemma \ref{lemma:SX}, the following  holds for each  $k \in \mc{K}^{d}$: \begin{align} \|F_{k}\|^{2}_{L^{2}(\RR^{d})} &= \Delta^{d}\sum\limits_{y\in X_{k} }|f(y) |^{2}\labeleq{fkl2}.
\end{align}

\noindent The norm of the iterates can then be bounded in the second step, for each $j'\in \mc{K}^{d-1}$,
\begin{align*} 
F_{j'}(\xi) &=(F^{-M}_{j'}(\xi), \hdots F^{M}_{j'}(\xi))\\
\|\tilde{C}_{j'}(\xi)\|_{\ell^{2}}^{2}& = \|VF_{j'}(\xi)\|_{\ell^{2}}^{2}\leq \|V\|_{2}^{2}\|F_{j'}(\xi)\|_{\ell^{2}}^{2} \\
 \|F_{j'}(\xi)\|_{\ell^{2}}^{2}&\leq  \|V^{-1}\tilde{C}_{j'}(\xi)\|_{\ell^{2}}^{2}\leq  \|V^{-1}\|_{2}^{2}\|\tilde{C}_{j'}(\xi)\|_{\ell^{2}}^{2}.
 \end{align*}
 Now, integrating over $\RR^{d}$,
 \begin{align*}
\|V^{-1}\|_{2}^{-2}\sum_{m_{d}\in\mc{M}^{d}}\|F_{j'}^{m_{d}}\|_{L^{2}(\RR^{d})}^{2}&\leq \sum_{j_{d} \in \mc{K}^{1}}\|F_{(j',j_{d})}\|_{L^{2}(\RR^{d})}^{2}\leq \|V\|_{2}^{2} \sum_{m_{d}\in\mc{M}^{d}}\|F_{j'}^{m_{d}}\|_{L^{2}(\RR^{d})}^{2}.
\end{align*}


\noindent In the third step, for $l=d-1, \hdots, 2$, for all $j'\in \mc{K}^{l-1}, m'\in \mc{M}^{d-l}$,
\begin{align*} 
F^{m'}_{j'}(\xi) &=(F^{(m',-M)}_{j'}(\xi), \hdots F^{(m',M)}_{j'}(\xi))\\
\|\tilde{C}_{j'}^{m'}(\xi)\|_{\ell^{2}}^{2}& = \|VF^{m'}_{j'}(\xi)\|_{\ell^{2}}^{2}\leq \|V\|_{2}^{2}\|F^{m'}_{j'}(\xi)\|_{\ell^{2}}^{2} \\
 \|F^{m'}_{j'}(\xi)\|_{\ell^{2}}^{2}&\leq  \|V^{-1}\tilde{C}^{m'}_{j'}(\xi)\|_{\ell^{2}}^{2}\leq  \|V^{-1}\|_{2}^{2}\|\tilde{C}^{m'}_{j'}(\xi)\|_{\ell^{2}}^{2} 
 \end{align*}
 Now, integrating over $\RR^{d}$,
 \begin{align*}
 \|V^{-1}\|_{2}^{-2}\sum_{m_{l}\in\mc{M}^{1}}\|F_{j'}^{(m_{l},m')}\|_{L^{2}(\RR^{d})}^{2}&\leq \sum_{j_{l} \in \mc{K}^{1}}\|F^{m'}_{(j',j_{l})}\|_{L^{2}(\RR^{d})}^{2}\leq \|V\|_{2}^{2} \sum_{m_{l}\in\mc{M}^{1}}\|F_{j'}^{(m_{l}, m')}\|_{L^{2}(\RR^{d})}^{2}.
\end{align*}
\noindent In Step 4, corresponding to each $ m'\in \mc{M}^{d-1}$, the resulting estimates are 
\begin{align} 
 \|V^{-1}\|_{2}^{-2} \sum_{m_{1}\in \mc{M}^{1}}\|c_{(m_{1},m')}\|_{L^{2}(\RR^{d})}^{2}&\leq \sum_{j_{1} \in \mc{K}^{1}}\|F^{m'}_{j_{1}}\|_{L^{2}(\RR^{d})}^{2}\leq \|V\|_{2}^{2} \sum_{m_{1}\in \mc{M}^{1}}\|c_{(m_{1},m')}\|_{L^{2}(\RR^{d})}^{2} \labeleq{cmest}.
\end{align} 
\noindent Combining \refeq{fkl2} and \refeq{cmest}, and the reconstruction step,
\begin{align*} 
 \|V^{-1}\|_{2}^{-2d} \|f\|_{L^{2}(\RR^{d})}^{2}&\leq \Delta^{d} \sum_{y \in X}|f(y)|^{2}\leq \|V\|_{2}^{2d} \|f\|_{L^{2}(\RR^{d})}^{2} \labeleq{cmest}.
\end{align*}

The upper bound in \refeq{thmab} is then given by $B=\|V\|_{2}^{2d}  $ and $(2M+1)\leq \|V\|^{2}_{2}\leq(2M+1)^{2}$ \cite{Pan2005}. The lower bound in \refeq{thmab} is then given by $A=\|V^{-1}\|_{2}^{-2d}$. Note that since $1=\|VV^{-1}\|_{2}\leq \|V\|_{2}\|V^{-1}\|_{2}$, it must follow that $A\leq B$. Using the identity $\|V^{-1}\|_{2}^{2} = (2M+1)\|V^{-1}\|_{\infty}^{2}$ and  Lemma \ref{lem:vinv} results in the estimate on the lower bound of $A$ in \refeq{thmab}. In the case $N=1$ and $\delta=1/(2M+1)$, Lemma \ref{lem:vinv} gives  $\|V^{-1}\|_{\infty}=1$, and then $A=B=(2M+1)^{d}$.

\end{proof}
\section*{Acknowledgements} This work has benefited from valuable discussions and insights from Bjorn Engquist, Kasso Okoudjou, David Walnut, Haomin Zhou, and an anonymous reviewer. This research was supported in part by NSF grants DMS-1317015, DMS-1720306,  and Institut Mittag-Leffler.


\end{document}